\documentclass[10pt]{amsart}
\usepackage{amsmath,amsthm,amssymb}
\usepackage{graphicx}
\usepackage[margin=1.3in]{geometry}

\newcommand{\R} {{\mathbb R}}                              
\newcommand {\K} {\mathcal{K}}                             
\newcommand {\di} {\Delta}                             
\newcommand {\w} {\omega}                             
\newcommand {\tr} {\Delta}                             
\newcommand {\sg} {\gamma}                             
\newcommand{\eps} {\varepsilon}
\newcommand {\dpos} {{\mathcal D}}

\newcommand{\hide}[1]{}

\newcommand{\vis}{\mathcal{V}}
\newcommand{\vispos}{\mathfrak{V}_n}

\newcommand{\spaceg}{G_S}
\newcommand{\spacep}{\mathfrak{P}_S}

\theoremstyle{plain}
\newtheorem{thm}{Theorem}
\newtheorem{prop}[thm]{Proposition}
\newtheorem{cor}[thm]{Corollary}
\newtheorem{lem}[thm]{Lemma}
\newtheorem*{opp}{Open Problem}
\newtheorem*{dop}{Deformation Problem}

\theoremstyle{definition}
\newtheorem{defn}[thm]{Definition}
\newtheorem*{exmp}{Example}

\theoremstyle{remark}
\newtheorem*{rem}{Remark}

\numberwithin{equation}{section}

{\makeatletter
 \gdef\xxxmark{%
   \protected@write\@auxout{\def\PAGE{ page }}
     {\@percentchar xxx: section \thesubsubsection \PAGE \thepage}%
   \expandafter\ifx\csname @mpargs\endcsname\relax 
     \expandafter\ifx\csname @captype\endcsname\relax 
       \marginpar{xxx}
     \else
       xxx 
     \fi
   \else
     xxx 
   \fi}
 \gdef\xxx{\@ifnextchar[\xxx@lab\xxx@nolab}
 \long\gdef\xxx@lab[#1]#2{{\bf [\xxxmark #2 ---{\sc #1}]}}
 \long\gdef\xxx@nolab#1{{\bf [\xxxmark #1]}}
 \gdef\turnoffxxx{\long\gdef\xxx@lab[##1]##2{}\long\gdef\xxx@nolab##1{}}%
}

\begin{document}

\title {Visibility graphs and deformations of associahedra}


\author{Satyan L.\ Devadoss}
\address{S.\ Devadoss: Williams College, Williamstown, MA 01267}
\email{satyan.devadoss@williams.edu}

\author{Rahul Shah}
\address{R.\ Shah: Williams College, Williamstown, MA 01267}
\email{rahul.a.shah@williams.edu}

\author{Xuancheng Shao}
\address{X.\ Shao: MIT, Cambridge, MA 02139}
\email{zero@mit.edu}

\author{Ezra Winston}
\address{E.\ Winston: Bard College, Annandale-on-Hudson, NY 12504}
\email{ew429@bard.edu}

\begin{abstract}
The associahedron is a convex polytope whose face poset is based on nonintersecting diagonals of a convex polygon.  In this paper, given an arbitrary simple polygon $P$, we construct a polytopal complex analogous to the associahedron based on convex diagonalizations of $P$.  We describe topological properties of this complex and provide realizations based on secondary polytopes.  Moreover, using the visibility graph of $P$, a deformation space of polygons is created which encapsulates substructures of the associahedron.
\end{abstract}


\maketitle


\baselineskip=16pt

%
%
\section{Associahedra from Polygons}
\subsection{}

Let $P$ be a simple planar polygon with labeled vertices.  Unless mentioned otherwise, assume the vertices of $P$ to be in general position, with no three collinear vertices.  A \emph{diagonal} of $P$ is a line segment connecting two vertices of $P$ which is contained in the interior of $P$.  A \emph{diagonalization} of $P$ is a partition of $P$ into smaller polygons using noncrossing diagonals of $P$.  Let a \emph{convex} diagonalization of $P$ be one which divides $P$ into smaller convex polygons. 

\begin{defn} \label{d:poset}
Let $\pi(P)$ be the poset of all convex diagonalizations of $P$ where for $a \prec a'$ if $a$ is obtained from $a'$ by adding new diagonals.\footnote{Mention of \emph{diagonals} will henceforth mean noncrossing ones.}
\end{defn}

\noindent Figure~\ref{f:refine_pic} shows diagonalizations of a polygon $P$.  Parts (b) through (d) show some elements of $\pi(P)$ where part (c) is greater than (d) in the poset ordering.

\begin{figure}[h] 
\includegraphics[width=.9\textwidth]{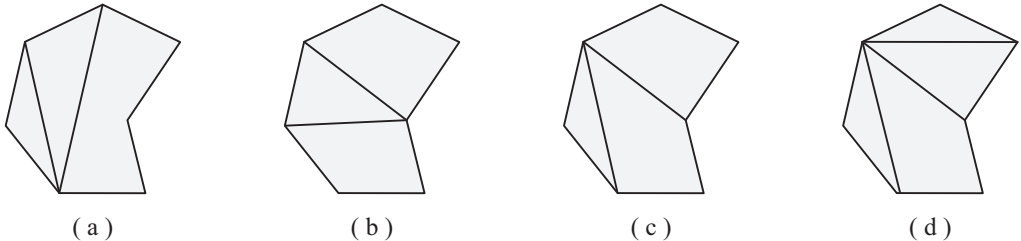}
\caption{A polygon with (a) a diagonalization and (b) -- (d) convex diagonalizations.}
\label{f:refine_pic}
\end{figure}

If $P$ is a convex polygon, all of its diagonalizations will obviously be convex.  The question was asked in combinatorics as to whether there exists a convex polytope whose face poset is isomorphic to $\pi(P)$ for a convex polygon $P$.  It was independently proven in the affirmative by Lee \cite{lee} and Haiman (unpublished): 
 
\begin{thm} \label{t:convexass}
When $P$ is a convex polygon with $n$ sides, the \emph{associahedron} $\K_n$ is a convex polytope of dim $n-3$ whose face poset is isomorphic to $\pi(P)$.
\end{thm}

\noindent Almost twenty years before this result was discovered, the associahedron had originally been defined by Stasheff for use in homotopy theory in connection with associativity properties of $H$-spaces \cite{sta}.  Associahedra have continued to appear in a vast number of mathematical fields, currently leading to numerous generalizations (see \cite{cd}  \cite{cfz} for some viewpoints).   Figure~\ref{f:k5} shows the three-dimensional polyhedron $\K_6$ on the left.  The 1-skeleton of $\K_6$ is shown on the right, with its vertices labeled with the appropriate diagonalizations of the convex hexagon.

\begin{rem}
Classically, the associahedron is based on all bracketings of $n-1$ letters and denoted as $K_{n-1}$.  Since there is a bijection between bracketings on $n-1$ letters and diagonalizations of $n$-gons, we use the script notation $\K_n$ with an index shift to denote the associahedron $K_{n-1}$ for ease of notation in our polygonal context.
\end{rem}

\begin{figure}[h]
\includegraphics[width=.9\textwidth]{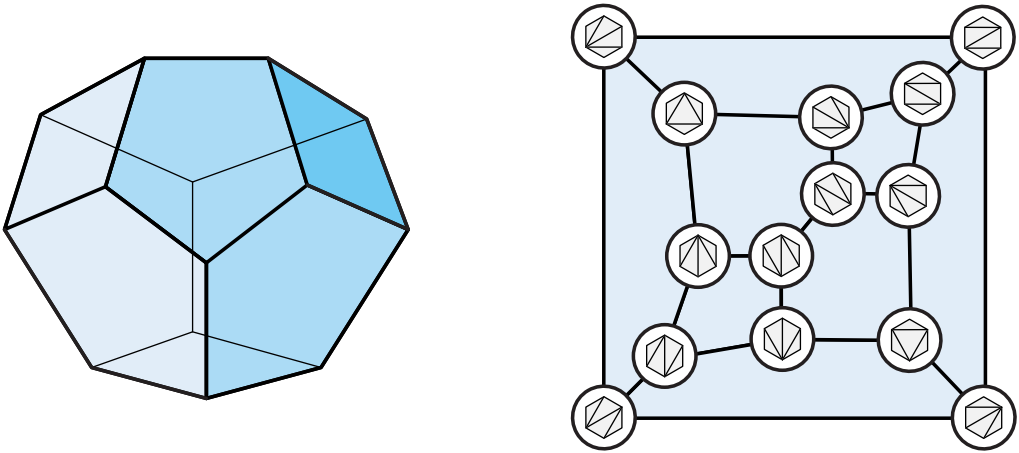}
\caption{The 3D associahedron $\K_6$.}
\label{f:k5}
\end{figure}

\subsection{}

We now extend Theorem~\ref{t:convexass} for arbitrary simple polygons $P$.  A \emph{polytopal complex} $S$ is a finite collection of convex polytopes (containing all the faces of its polytopes) such that the intersection of any two of its polytopes is a (possibly empty) common face of each of them. The dimension of the complex $S$ denoted as $\dim(S)$ is the largest dimension of a polytope in $S$.

\begin{thm} \label{t:subcomplex}
For a polygon $P$ with $n$ vertices, there exists a polytopal complex $\K_P$ whose face poset is isomorphic to $\pi(P)$.  Moreover, $\K_P$ is a subcomplex of the associahedron $\K_n$.
\end{thm}

\begin{proof}
Let $p_1, \ldots, p_n$ be the vertices of $P$ labeled cyclically.  For a convex $n$-gon $Q$, let  $q_1, \ldots, q_n$ be its vertices again with clockwise labeling. The natural mapping from $P$ to $Q$ (taking $p_i$ to $q_i$) induces an injective map $\phi: \pi(P) \longrightarrow \pi(Q)$.  Assign to $t \in \pi(P)$ the face of $K_n$ that corresponds to $\phi(t) \in \pi(Q)$.  It is trivial to see that $\phi(t_1) \prec \phi(t_2)$ in $\pi(Q)$ if $t_1 \prec t_2$ in $\pi(P)$.

Moreover, for any $\phi(t)$ in $\pi(Q)$ and any diagonal $(q_i, q_j)$ which does not cross the diagonals of $\phi(t)$, we see that $(p_i, p_j)$ does not cross any diagonal of $t$.  
So if a face $f$ of $\K_n$ is contained in a face corresponding to $\phi(t)$, then there exists a diagonalization $t' \in \pi(P)$ where $\phi(t')$ corresponds to $f$ and $t' \prec t$.
Since the addition of any noncrossing diagonals to a convex diagonalization is still a convex diagonalization, the intersection of any two faces\footnote{Such an intersection could  possibly be empty if diagonals are crossing.} is also a face in $\K_P$. So $\K_P$ satisfies the requirements of a polytopal complex and (due to the map $\phi$) is a subcomplex of $\K_n$.
\end{proof}

\begin{cor}
Let $P$ be an $n$-gon and let $d(P)$ be the minimum number of diagonals required to diagonalize $P$ into convex polygons.  The polytopal complex $\K_P$ has dimension $n - 3 - d(P)$.
\end{cor}

\begin{proof}
The dimension of a polytopal complex is defined as the maximum dimension of any face. In the associahedron $\K_n$, a face of dimension $k$ corresponds to a convex diagonalization with $n-3-k$ diagonals.  The result follows since $\phi$ is an injection.
\end{proof}

\begin{exmp}
Figure~\ref{f:3d-6gons} shows two polytopal complexes $\K_P$ for the respective polygons $P$ given.  The left side is the 3-dimensional associahedron $\K_6$ based on diagonalizations of a convex hexagon, whereas the right side is the 2D polytopal complex of a deformed hexagon, made from two edges glued to opposite vertices of a square.  Note how this complex appears as a subcomplex of $\K_6$.

\begin{figure}[h]
\includegraphics[width=.9\textwidth]{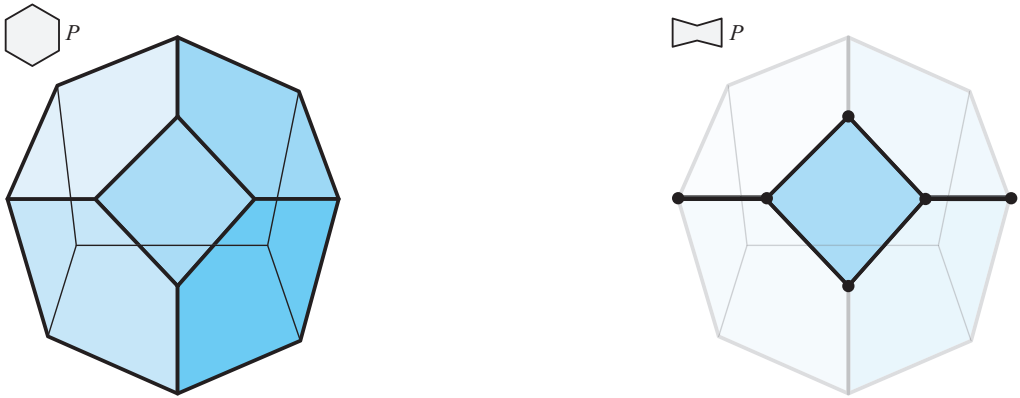}
\caption{Two polytopal complexes $\K_P$ based on the respective polygons $P$ given.}
\label{f:3d-6gons}
\end{figure}

\noindent The diagram given in Figure~\ref{f:6gonflip} shows the labeling of the complex given on the right side of Figure~\ref{f:3d-6gons}.  Note the number of diagonals in each diagonalization is constant across the dimensions of the faces.

\begin{figure}[h]
\includegraphics[width=.8\textwidth]{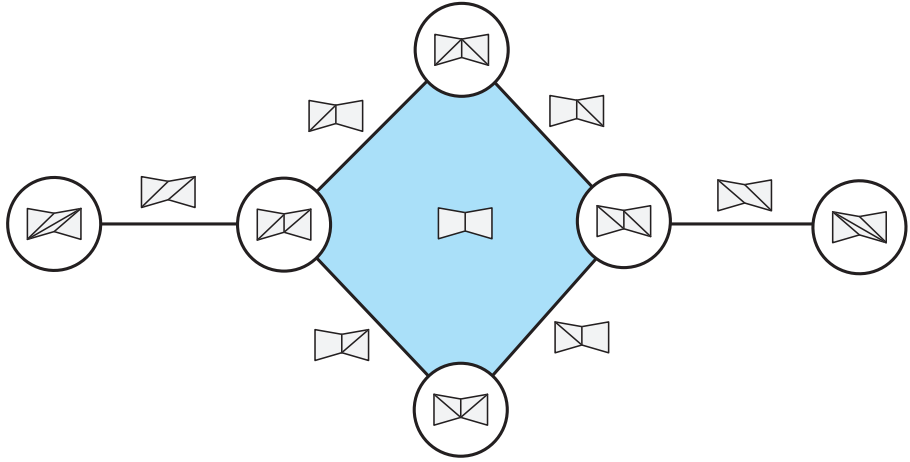}
\caption{The labeling of the face poset structure of $\K_P$ given on the right side of Figure~\ref{f:3d-6gons}.}
\label{f:6gonflip}
\end{figure}
\end{exmp}

\subsection{} \label{ss:altconstruct}

An alternate construction of $\K_P$ comes from \emph{removing} certain faces of $\K_n$:  Each \emph{facet} of $\K_n$ corresponds to a diagonal of $Q$.  Now consider the set $d_{Q \setminus P}$ of diagonals of $Q$ which are not diagonals of $P$.  If a facet $f$ of $\K_n$ corresponds to a diagonal of $d_{Q \setminus P}$,  remove $f$ along with the interior of any face $g$ where $f \cap g$ is nonempty.  Notice that this deletes every face that does not correspond to a convex diagonalization of $P$ while preserving all faces that do. It is easy to see that we are left with a polytopal complex $\K_P$.  Since for every two intersecting diagonals there is a third diagonal intersecting neither, any two facets are separated by at most one facet.  We therefore have the following:

\begin{lem}
The subcomplex of $\K_n$ which is removed to form $\K_P$ is connected.
\end{lem}

Although the complement of $\K_P$ in $\K_n$ is connected, it is not immediate that $\K_P$ itself is connected.  Consider a diagonalization of a polygon $P$.  An \emph{edge flip} of $P$  (called \emph{flip} for short) removes a diagonal of $P$ and replaces it with another noncrossing diagonal of $P$.  The \emph{flip graph} of $P$ is a graph whose nodes are the set of triangulations of $P$, where two nodes $x$ and $y$ are connected by an arc if one diagonal of $x$ can be flipped to obtain $y$.  It is a classical result of computational geometry that the flip graph of any polygon is connected; see \cite{bh} for an overview.  Since the flip graph of $P$ is simply the 1-skeleton of $\K_P$, we obtain:

\begin{thm}
$\K_P$ is connected for any $P$.
\end{thm} 

\noindent The right side of Figure~\ref{f:k5} along with Figure~\ref{f:6gonflip} display flip graphs of the appropriate polygons.

%
%
\section{Topological Properties}
\subsection{}

We begin this section by considering arbitrary (not just convex) diagonalizations of $P$ and the resulting geometry of $\K_P$.  Let $\di = \{d_1, \ldots, d_k\}$ be a set of noncrossing diagonals of $P$, and let $\K_P(\di)$ be the collection of faces in $\K_P$ corresponding to all diagonalizations  of $P$ containing $\di$. 

\begin {lem}
$\K_P(\di)$ is a polytopal complex.
\end{lem}

\begin {proof}
If $t$ is a diagonalization of $P$ containing $\di$, then any $t'$ in $\pi(P)$ must also contain $\di$ if $t' \prec t$.  Thus there must be a face in $\K_P(\di)$ that corresponds to $t'$.  Furthermore, consider faces $f_1$ and $f_2$ in $\K_P(\di)$ corresponding to diagonalizations $t_1$ and $t_2$ of $P$. Then the intersection of $f_1$ and $f_2$ must correspond to a diagonalization including $\di$ since $f_1 \cap f_2$ is the (possibly empty) face corresponding to all convex diagonalizations that include every diagonal of
$t_1$ and $t_2$.
\end{proof}

The following lemma is an immediate consequence of the construction of $\K_P(\di)$.  It shows the gluing map between two faces of $\K_P$.

\begin {lem} \label {l:glue}
Let $\di_1$ and $\di_2$ be two collections of noncrossing diagonals of $P$.  Then $\K_P(\di_1)$ and $\K_P(\di_2)$ are glued together in $\K_P$ along the (possibly empty) polytopal subcomplex $\K_P(\di_1 \cup \di_2)$.
\end{lem}

\begin{thm} \label{t:product}
Suppose the diagonals $\di = \{d_1, \ldots, d_k\}$ divide $P$ into polygons $Q_0, \ldots, Q_k$. Then $\K_P(\di)$ is isomorphic to the cartesian product ${\K}_{Q_0} \times \cdots \times {\K}_{Q_k}.$
\end{thm}

\begin{proof}
We use induction on $k$. When $k = 1$, any face $f \in \K_P(d)$ corresponds\footnote{We abuse notation by writing $\K_P(d) = \K_P(\{d\})$.} to a convex diagonalization of $Q_0$ paired with a convex diagonalization of $Q_1$.  Thus, a face of ${\K}_{Q_0} \times {\K}_{Q_1}$ exists for each pair of faces $(f_0,f_1)$, for $f_0 \in {\K}_{Q_0}$ and $f_1 \in {\K}_{Q_1}$.  For $k > 1$, order the diagonals such that $d_{k}$ divides $P$ into polygons $Q_* = Q_0\cup\cdots\cup Q_{k-1}$ and $Q_{k}$. A face in $\K_P(\di)$ corresponds to a convex diagonalization $t_1$ of $Q_*$ and a convex diagonalization $t_2$ of $Q_{k}$. The pair $(t_1,t_2) \in \K_P(\di)$ corresponds to a face in $\K_{Q_*}(\di \setminus d_k) \times \K_{Q_{k}}$. By the induction hypothesis, $\K_{Q_*}(\di \setminus d_k)$ is isomorphic to 
${\K}_{Q_0} \times {\K}_{Q_1} \times \ldots {\K}_{Q_{k-1}}.$
\end{proof}

\begin{cor} \label{c:prodass}
If the diagonals $\di$ divide $P$ into \emph{convex} polygons with  $Q_0, \ldots, Q_k$ where $Q_i$ has $n_i$ edges, then $\K_P(\di)$ is the product of \emph{associahedra} $\K_{n_0} \times \cdots \times \K_{n_k}$. 
\end{cor}

\noindent Thus every face of $\K_P$ is a product of associahedra.  For a polytopal complex $\K_P$, the \emph{maximal} elements of its face poset $\pi(P)$ are analogous to facets of convex polytopes.  These elements of $\pi(P)$ are characterized by the following:

\begin{defn}
A face $f$ of $\K_P$ corresponding to a diagonalization $t \in \pi(P)$ is a \emph{maximal face} if there does not exist $t' \in \pi(P)$ such that $t \prec t'$.
\end{defn}

\noindent Thus a maximal face of $\K_P$ has a convex diagonalization of $P$ using the \emph{minimal} number of diagonals.  Figure~\ref{f:mindiags} shows a polygon $P$ along with six minimal convex diagonalizations of $P$.  As this shows, such diagonalizations may not necessarily have the same number of diagonals.

\begin{figure}[h]
\includegraphics[width=\textwidth]{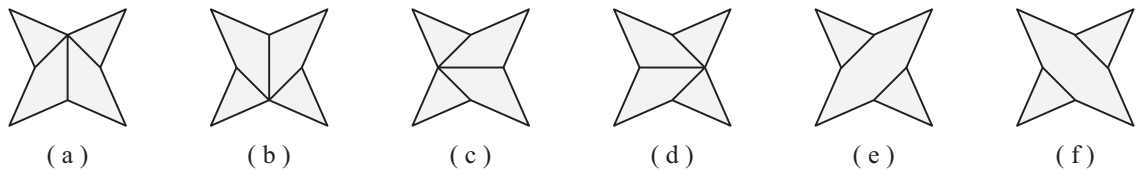}
\caption{Six minimal convex diagonalizations of a polygon.}
\label{f:mindiags}
\end{figure}

Corollary~\ref{c:prodass} shows that each maximal face is a product of associahedra.  Moreover, by Lemma~\ref{l:glue}, the maximal faces of $\K_P$ can be glued together to construct $\K_P$.   Figure~\ref{f:superpic} shows an example of $\K_P$ for the polygon shown.  It is a polyhedral subcomplex of the 5-dimensional convex associahedron $\K_8$.  We see that $\K_P$ is made of six maximal faces, four squares (where each square is a product $\K_4 \times \K_4$ of line segments) and two $\K_6$ associahedra.  Each of these six faces correspond to the minimal convex diagonalizations of $P$ given in Figure~\ref{f:mindiags}.

\begin{figure}[h] 
\includegraphics[width=.8\textwidth]{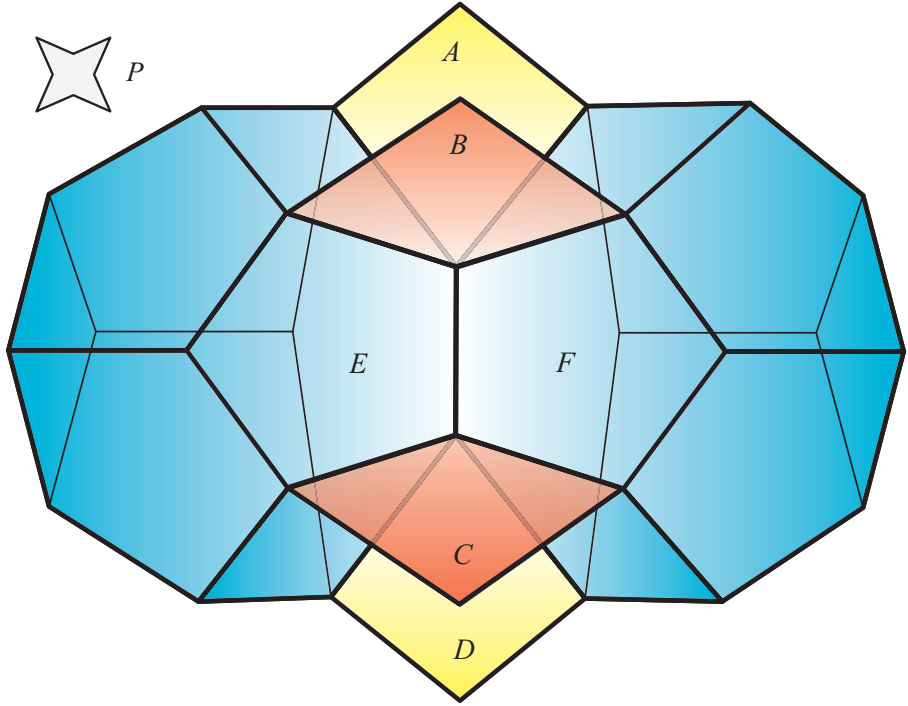}
\caption{An example of $\K_P$ seen as a 3-dimensional subcomplex of the 5-dimensional $\K_8$.  It is made of six maximal faces, four $\K_4 \times \K_4$ squares and two $\K_6$ associahedra.}
\label{f:superpic}
\end{figure}

\subsection{}

Having already shown that $\K_P$ is connected, we now prove that this polytopal complex is indeed contractible.  We start with some basic geometry:  A vertex of a polygon is called \emph{reflex} if the diagonal between its two adjacent vertices cannot exist.  Note that every nonconvex polygon has a reflex vertex.

\begin {lem} \label{l:reflexexistence}
For any reflex vertex $v$ of a nonconvex polygon $P$, every element of $\pi(P)$ has at least one diagonal incident to $v$.
\end {lem}

\begin {proof}
Assume otherwise and consider an element of $\pi(P)$.  In this convex diagonalization, since there is no diagonal incident to $v$, there exists a unique subpolygon containing $v$.  Since $v$ is reflex, this subpolygon cannot be convex, which is a contradiction.
\end {proof}

\begin {lem} \label{l:collection}
Let $F = \{f_1, f_2, \ldots, f_k\}$  be a set of faces of $\K_P$ such that $\bigcap_{F} f_i$ is nonempty.  If  $\bigcap_{F'} f_i$ is contractible for \emph{every} $F' \subset F$, then $\bigcup_{F} f_i$ is contractible.
\end{lem}

\begin{proof}
We prove this by induction on the number of faces in $F$. A single face is trivially contractible. Now assume  $\bigcup_{F} f_i$ is contractible.  For a face $f_{k+1} \notin F$ of $\K_P$, let $G = \{f_{k+1}\} \cup F$ so that $\bigcap_{G} f_i$ is nonempty. Since $f_{k+1}$ intersects the intersection of the faces $F$ and since this intersection is nonempty and contractible, we can deformation retract $\bigcup_G f_i$ onto $\bigcup_F f_i$ and hence maintain contractibility.
\end{proof}

\begin {thm} \label{t:contract}
For any polygon $P$, the polytopal complex $\K_P$ is contractible.
\end{thm}

\begin {proof}
We prove this by induction on the number of vertices.  For the base case, note that $\K_P$ is a point for any triangle $P$.  Now let $P$ be a polygon with $n$ vertices.   If $P$ is convex, then $\K_P$ is the associahedron $\K_n$ and we are done.  For $P$ nonconvex, let $v$ be a reflex vertex of $P$.  Since each diagonal $d$ of $P$ incident to $v$  seperates $P$ into two smaller polygons $Q_1$ and $Q_2$, by our hypothesis, $\K_{Q_1}$ and $\K_{Q_2}$ are contractible.  Theorem \ref{t:product} shows that $\K_P(d)$ is isomorphic to $\K_{Q_1} \times \K_{Q_2}$, resulting in $\K_P(d)$ to be contractible.

Let $\di = \{d_1, \ldots, d_k\}$ be the set of all diagonals incident to $v$.  Since $\di$ is a set of \emph{noncrossing} diagonals, then $\bigcap_{\di} \K_P(d)$ is nonempty.  Furthermore, for any subset $\di' \subset \di$, we have $\K_P(\di') = \bigcap_{\di'} \K_P(d).$  By Theorem \ref{t:product}, this is a product of contractible pieces, and thus itself is contractible.  Therefore, by Lemma \ref{l:collection}, the union of the complexes $\K_P(d_i)$ is contractible.   However, since Lemma \ref{l:reflexexistence} shows that this union is indeed $\K_P$, we are done.  
\end{proof}

\begin{rem}
A recent paper by Braun and Ehrenborg \cite{be} considers a similar complex for nonconvex polygons $P$.  Their complex $\theta(P)$ is a \emph{simplicial} complex with vertex set the diagonals in $P$ and facets given by triangulations of $P$.  Indeed, it is easy to see that $\theta(P)$ is the combinatorial dual of $\K_P$.  Figure~\ref{f:be-dual} shows an example of $\theta(P)$ on the left and $\K_P$ on the right, as depicted in Figure~\ref{f:6gonflip}.  
\begin{figure}[h]
\includegraphics[width=.9\textwidth]{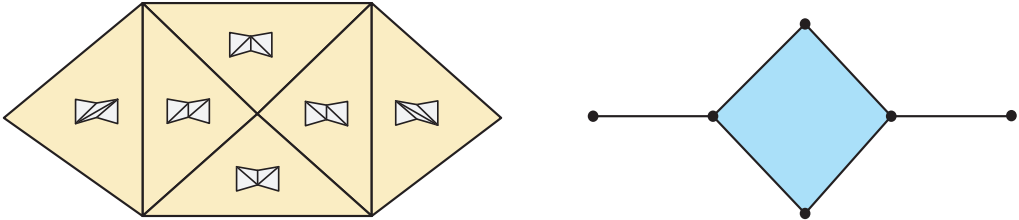}
\caption{The simplicial complex $\theta(P)$ and its corresponding dual $\K_P$.}
\label{f:be-dual}
\end{figure}
For a nonconvex polygon $P$ with $n$ vertices, the central result in \cite{be} shows that $\theta(P)$ is homeomorphic to a ball of dimension $n-4$.  This is, in fact, analogous to Theorem~\ref{t:contract} showing $\K_P$ is contractible.  The method used in \cite{be} is based on a \emph{pairing lemma} of Linusson and Shareshian \cite{sl}, motivated by discrete Morse theory, whereas our method considers the geometry based on reflex vertices.
\end{rem}

\subsection{}

The previous constructions and arguments can be extended to include planar polygons with holes.  These \emph{generalized polygons} are bounded, connected planar regions whose boundary is the disjoint union of simple polygonal loops. Vertices, edges, and diagonals are defined analogously to that for polygons.  In what follows, denote $R$ as a generalized polygon with $n$ vertices and $h+1$ boundary components.  Similar to Definition~\ref{d:poset}, one can define $\pi(R)$ to be the set of all convex diagonalizations of $R$ partially ordered by inclusion.  We state the following and leave the straight-forward proof\footnote{Recall the classical result that the number of diagonals in a triangulation of $R$ is $n+3h-3$.}    to the reader:

\begin{cor}
There exists a polytopal complex $\K_R$ whose face poset is isomorphic with $\pi(R)$.  Moreover, the dimension of $\K_R$ is $n+3h-d(R)-3$, where $d(R)$ is the minimum number of diagonals required to diagonalize $R$ into convex polygons.
\end{cor}

Moreover, since there is at least one reflex vertex in any generalized polygon (with more than one boundary component), the proof of the following is identical to that of Theorem~\ref{t:contract}:

\begin{cor}
$\K_R$ is contractible for a generalized polygon $R$. In particular, $\K_R$ is connected.
\end{cor}


\noindent Figure~\ref{f:polyhole} shows an example of the associahedron of a pentagon with a triangular hole, whose maximal faces are 8 cubes and 3 squares.
\begin{figure}[h]
\includegraphics[width=\textwidth]{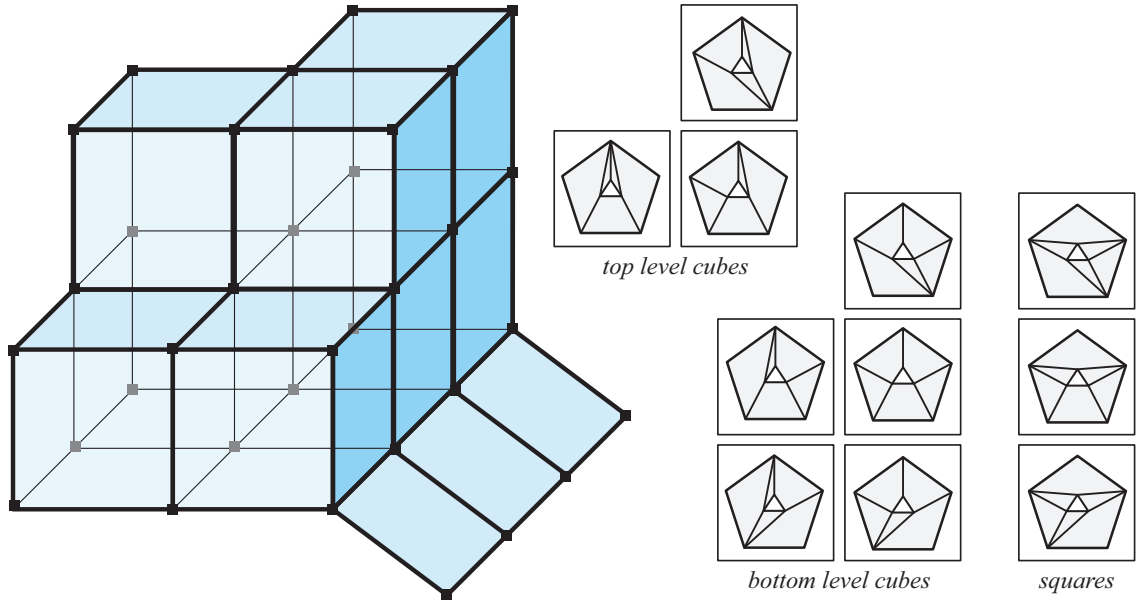}
\caption{The associahedron for a pentagon with a triangular hole, along with its 11 maximal faces.}
\label{f:polyhole}
\end{figure}
Similar to the polygonal case, the complexes $\K_R$ can be obtained by gluing together different polytopal complexes $\K_P$, for various polygons $P$.  For simplicity, we now describe this in the case when $R$ has only two boundary components.

Let $\di$ be the set of diagonals of $R$ connecting one vertex of the outside boundary region of $R$ with one vertex of the inside hole of $R$.  Observe that for any diagonal $d$ in $\di$, there exists a nonconvex polygon $P$ such that $\K_P$ is isomorphic to $\K_R(d)$.   The reason is that there is a bijection between all the convex diagonalizations of $R$ containing $d$ and all the convex diagonalizations of $P$ obtained by splitting the diagonal $d$ into two edges in the polygonal boundary, as given in Figure~\ref{f:holesplit}.  Since any convex diagonalization of $R$ contains at least one diagonal in $\di$, then $\K_R$ is constructed from gluing elements of $\{\K_R(d) \ | \ d \in \di\}$.  Now if the number of boundary components of $R$ is more than two, a similar argument as above can be applied to reduce the number of boundary components by one, and induction will yield the result.

\begin{figure}[h]
\includegraphics[width=.9\textwidth]{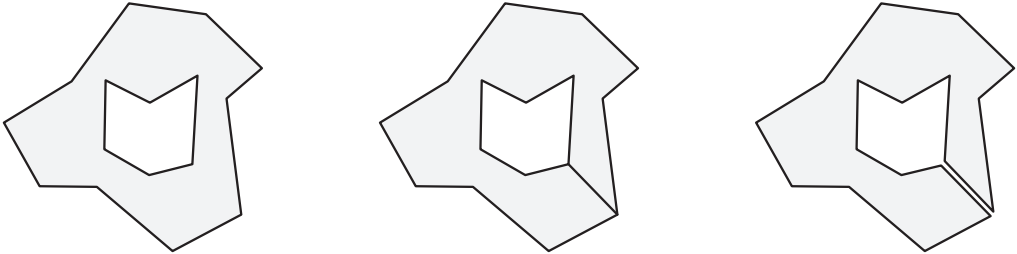}
\caption{A diagonal of a polygon with a hole leads to a nonconvex polygon without a hole.}
\label{f:holesplit}
\end{figure}

\begin{rem}
In order for an $n$-gon with a $k$-gon hole to \emph{not} have a polyhedral complex which is a subcomplex of $\K_{n+k+2}$, every diagonal from the hole to the exterior polygon must be optional (i.e. not required in every triangulation).  Otherwise, the construction given in Figure~\ref{f:holesplit} can be used to find such a map.  In particular, the complex given in Figure~\ref{f:polyhole} is not a subcomplex of $\K_{10}$.  Given any generalized polygon $R$, it is an interesting question to find the smallest $n$ such that $\K_R$ is a subcomplex of $\K_n$.
\end{rem}

%
%
\section{Geometric Realizations}
\subsection{}

This section focuses on providing a geometric realization of $\K_P$ with integer coordinates for its vertices.  There are numerous realizations of the classical associahedron \cite{lod} and its generalizations \cite{pos} \cite{cfz}.  We first recall the method followed in \cite{dev} for the construction of $\K_n$, altered slightly to conform to this paper.  This is introduced in order to contrast it with the secondary polytope construction which follows.

Let $P$ be a polygon and choose an edge $e$ of $P$ to be its base.  Any triangulation $T$ of $P$ has a dual structure as a rooted tree, where the root of the tree corresponds to the edge $e$.   For a triangle $\tr$ of $T$, let $r(\tr)$ be the collection of triangles of $T$ that have a path to the root (in the dual tree) containing $\tr$.  Figure~\ref{f:treeduality} gives three examples of the values $|r(\tr)|$ for triangles in different triangulations of  polygon; the triangle adjacent to the root is shaded.

\begin{figure}[h]
\includegraphics[width=.9\textwidth]{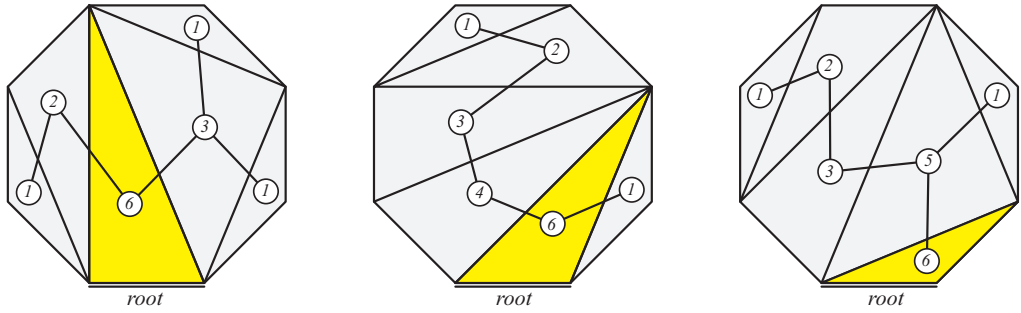}
\caption{Assigning values to triangles in different triangulations of a polygon.}
\label{f:treeduality}
\end{figure}

Define a value $\theta_T(\tr)$ for each triangle of $T$ using the recursive condition
\begin{equation}
\sum_{x \in r(\tr)} \theta_T(x) = 3^{|r(\tr)|-1}.
\label{e:integer}
\end{equation}
For each vertex $p_i$ of $P$, assign the value 
$$\widehat \theta_T(p_i) \ = \ \max \ \{ \ \theta_T(\tr) \ | \  p_i \in \tr \ \}.$$

\begin{thm} {\textup{\cite[Section 2]{dev}}} \label{t:realass}
Let $p_1, \ldots, p_n$ be vertices of a convex polygon $P$ such that $p_{n-1}$ and $p_n$ are the two vertices incident to the root edge $e$ of $P$.  The convex hull \footnote{Given a finite set of points $S$ in $\R^n$, we define the \emph{convex hull} of $S$ to be the smallest convex set containing $S$, whereas the \emph{hull} of $S$ is the boundary of the convex hull.} of the points
$$( \ \widehat \theta_T(p_1), \ \ldots, \ \widehat \theta_T(p_{n-2}) \ )$$
in $\R^{n-2}$, as $T$ ranges over all triangulations of $P$, yields the associahedron $\K_n$.
\end{thm}

Notice that since both of the vertices adjacent to the root edge will have the same value $\widehat \theta$ for all triangulations, we simply ignore them.  Moreover, it is easy to see from Eq.~\eqref{e:integer} that this realization of $\K_n$ must lie in the hyperplane $\sum x_i = 3^{n-3}$ of $\R^{n-2}$.  Figure~\ref{f:realass} shows the values  $\widehat \theta$ given to each vertex in each triangulation, along with the convex hull of the five points.

\begin{figure}[h]
\includegraphics[width=.9\textwidth]{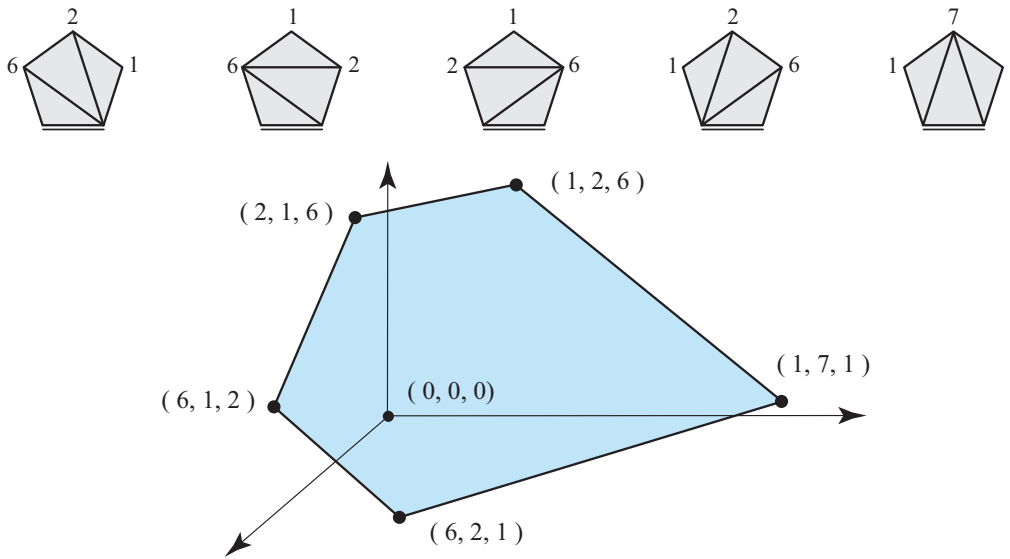}
\caption{An example of Theorem~\ref{t:realass} producing $\K_5$.}
\label{f:realass}
\end{figure}

\begin{rem}
For nonconvex polygons, since $\K_P$ is a subcomplex of $\K_n$ by Theorem~\ref{t:subcomplex}, the method above extends to a realization of $\K_P$ with integer coordinaters.
\end{rem}

\subsection{}

The notion of a secondary polytope was developed by Gelfand, Kapranov, and Zelevinsky \cite{gkz}.  We review it and show its applicability to $\K_P$.    Let $P$ be a polygon with vertices $p_1, \ldots, p_n$.  For a triangulation $T$ of $P$, let
$$\phi(p_i) \ = \ \sum_{p_i \in\Delta\in T} \text{area}(\Delta) $$
be the sum of the areas of all triangles $\Delta$ which contain the vertex $p_i$. Let the \emph{area vector} of $T$ be 
$$\Phi(T) \ = \ (\phi(p_1),\ldots,\phi(p_n)).$$

\begin{defn}
The \emph{secondary polytope} of $\Sigma(P)$ of a polygon $P$ is the convex hull of the area vectors of all triangulations of $P$. 
\end{defn} 

The following is a special case of the work on secondary polytopes:

\begin{thm}  {\textup{\cite[Chapter 7]{gkz}}}
If $P$ is a convex $n$-gon, the secondary polytope $\Sigma(P)$ is a realization of the associahedron $\K_n$.
\end{thm}

\noindent It is easy to see that all the area vectors of $P$ lie in an $(n-3)$-dimensional plane of $\R^n$.  However, it is not immediate that each area vector  is on the hull of $\Sigma(P)$.  

We show that the secondary polytope of a nonconvex polygon has all its area vectors on its hull, as is the case for a convex polygon.  However, since the secondary polytope of a nonconvex polygon is \emph{not} a subcomplex of the secondary polytope for convex polygon, this result is not trivial.\footnote{The original proof for the convex case uses a \emph{regular} triangulation observation; we modify this using a lifting map of the dual tree of the polygon.}

\begin{thm}
For any polygon $P$, all area vectors $\Phi(T)$ lie on the hull of $\Sigma(P)$.
\end{thm}

\begin{proof}
Fix a triangulation $T$ of $P$.  We first show that there is a height function $\w$ on $P$ which raises the vertices of $T$  to a locally convex surface in $\mathbb{R}^3$, that is, a surface which is convex on every line segment in $P$.  Choose an edge $e$ of $P$ to be its base so that the dual tree of $T$ is rooted at $e$. Starting from the root and moving outward, assign increasing numbers $m_i$ to each consecutive triangle $\tr_i$ in the tree. Define a height function
$$\w(p_i) \ = \ \min \ \{m_k \ | \ p_i \in \tr_k \ \}$$
for each vertex $p_i$ of $P$.
Observe that for every pair of adjacent triangles $\tr_1$ and $\tr_2$ (in the dual tree), we can choose the value $m_i$ to be large enough such that the planes containing $\w(\tr_1)$ and $\w(\tr_2)$ are distinct and meet in a convex angle; see Figure~\ref{f:height} below.

\begin{figure}[h]
\includegraphics[width=.8\textwidth]{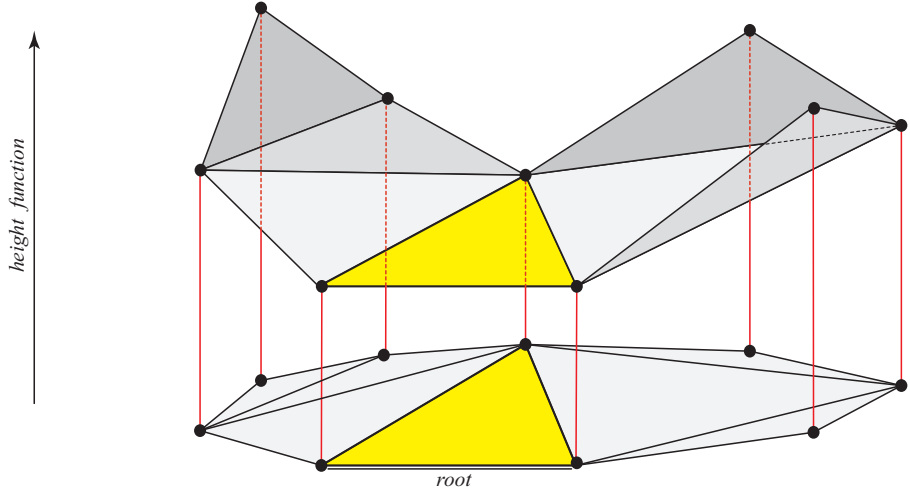}
\caption{A height function $\w$ lifting the triangles of $T$.}
\label{f:height}
\end{figure}

Now, in order to show that $\Phi(T)$ lies on the hull of $\Sigma(P)$, we construct a linear function $\rho(v)$ on $\Sigma(P)$ such that $\rho(\Phi(T))$ is a unique minimum of this function on $\Sigma(P)$.  For any $v$ in $\Sigma(P)$, define $\rho(v) = \langle \w(T), v \rangle$ to be the inner product of the vectors $v \in \R^n$ and 
$$\w(T) = (\w(p_1), \ldots, \w(p_n)).$$ For a triangle $\tr$ of $T$ with vertices $p_i, p_j, p_k$, consider the volume in $\R^3$ enclosed between $\tr$ and the lifted triangle $\w(\tr)$. This volume can be written as
$$\frac{\w(p_i) + \w(p_j) + \w(p_k)}{3} \ \text{area}(\tr).$$
The volume between the surface on which the $\w(p_i)$'s lie and the plane is given by
\begin{align*}
\sum_{\tr \in T} \ \frac{\w(p_i) + \w(p_j) + \w(p_k)}{3} \ \text{area}(\tr) \ =&  \ \sum_{i=1}^n \left[\frac{\w(p_i)}{3}\sum_{p_i\in\Delta\in T}\text{area}(\tr)\right]  \\
=& \ \sum_{i=1}^n \ \frac{\w(p_i)}{3} \ \phi(p_i) \\
=& \ \frac{1}{3} \ \langle \w(T), \Phi(T) \rangle.
\end{align*}
Since $\w$ lifts $T$ to a locally convex surface $S$, we know that $w$ will lift any $T'\neq T$ to a surface $S'$ above $S$. Thus $\langle \w(T), \Phi(T) \rangle <\langle \w(T), \Phi(T') \rangle$, implying all vertices of $\Sigma(P)$ lie on the hull.
\end{proof}

\begin{cor}
If $P$ is nonconvex, then a subset of the faces of $\Sigma(P)$ yield a realization of $\K_P$.
\end{cor}

\begin{proof}
For any face $f$ of $\K_P$, let $T_1, \ldots, T_k$ be the triangulations corresponding to the vertices of $f$.  We use the same argument as the theorem above to show there exists a height function $\w$  such that $\langle \w, \Phi(T) \rangle$ is constant for any $T \in \{T_1,\ldots,T_k\}$ and  $\langle \w, \Phi(T) \rangle < \langle \w, \phi(T') \rangle$ for any $T' \notin \{ T_1,\ldots, T_k \}$.
\end{proof}

%
%
\section{Visibility Graphs}
\subsection{}

We now consider the space of simple planar polygons through deformations.  Throughout this section, we only consider simple polygons with vertices labeled $\{1, 2, \ldots, n\}$ in this cyclic order.  As before, we assume the vertices of $P$ to be in general position, with no three collinear vertices.    Let us begin with the notion of visibility. 
 
\begin{defn}
The \emph{visibility graph} $\vis(P)$ of a labeled polygon $P$ is the labeled graph with the same vertex set as $P$, with $e$ as an edge of $\vis(P)$ if $e$ is an edge or diagonal of $P$.
\end{defn}

A classical open problem in visibility of polygons is as follows:

\begin{opp}
Given a graph $G$, find nice necessary and sufficient conditions which show if there exists a simple polygon $P$ such that $\vis(P) = G$.
\end{opp}

\begin{defn}
Two polygons $P_1$ and $P_2$ are \emph{$\vis$-equivalent} if  $\vis(P_1)=\vis(P_2)$.
\end{defn}

There is a natural relationship between the graph $\vis(P)$ and the polytopal complex $\K_P$:  If  polygons $P_1$ and $P_2$ are $\vis$-equivalent then  $\K_{P_1}$ and $\K_{P_2}$ yield the same complex.  We wish to classify polygons under a stronger relationship than $\vis$-equivalence.  For a polygon $P$, let $(x_i,y_i)$ be the coordinate of its $i$-th vertex in $\R^2$. We associate a point $\sg(P)$ in $\R^{2n}$ to $P$ where
$$\sg(P) \ = \ (x_1,y_1,x_2,y_2,\ldots,x_n,y_n).$$
Since $P$ is labeled, it is obvious that $\sg$ is injective but not surjective. 

\begin{defn}
Two polygons $P_1$ and $P_2$ are \emph{$\vis$-isotopic} if there exists a continuous map $f:[0,1]\longrightarrow \R^{2n}$ such that $f(0)=\sg(P_1)$, $f(1)=\sg(P_2)$, and for every $t\in [0,1]$, $f(t)=\sg(P)$ for some simple polygon $P$ where $\vis(P) = \vis(P_1)$.
\end{defn}

\noindent It follows from the definition that two polygons that are $\vis$-isotopic are $\vis$-equivalent.  The converse is not necessarily true:  Figure~\ref{f:visequivalent} shows three polygons along with their respective visibility graphs.  Parts (a) and (c) are $\vis$-equivalent (having identical visibility graphs) but cannot be deformed into one another without changing their underlying visibility graphs.  The middle figure (b) shows an intermediate step in obtaining a deformation.

\begin{figure}[h]
\includegraphics[width=\textwidth]{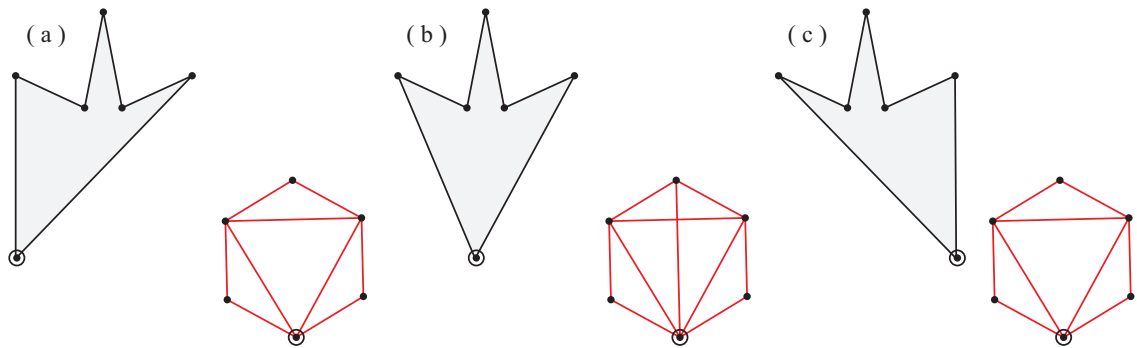}
\caption{Polygons (a) and (c) are $\vis$-equivalent but not $\vis$-isotopic.  The visibility graphs for each of the polygons are shown.}
\label{f:visequivalent}
\end{figure}

For a polygon $P$ with $n$ vertices, let $\dpos(P)$  be the $\vis$-isotopic equivalence class containing the polygon $P$ and let $\dpos$ be the set of all such equivalence classes of polygons with $n$ vertices. We give $\dpos$ a poset structure:  For two $n$-gons $P_1$ and $P_2$, the relation $\dpos(P_2) \prec \dpos(P_1)$ is given if the following two conditions hold:
\begin{enumerate}
\item $\vis(P_1)$ is obtained by adding one more edge to $\vis(P_2)$.
\item There exists a continuous map $f:[0,1]\longrightarrow \R^{2n}$, such that $f(0)=\sg(P_1)$, $f(1)=\sg(P_2)$, and for every $t\in [0,1/2)$, $f(t)=\sg(P)$ for some polygon $P$ with $\vis(P)=\vis(P_1)$, while for every $t\in (1/2,1]$, $f(t)=\sg(Q)$ for some polygon $Q$ with $\vis(Q)=\vis(P_2)$.
\end{enumerate}
If $P_1$ and $P_2$ are $\vis$-isotopic, let $\dpos(P_1) = \dpos(P_2)$.  Taking the transitive closure of $\preceq$ yields the \emph{deformation poset} $\dpos$.  A natural ranking exists on $\dpos$ based on the number of edges of the visibility graphs.

\begin{exmp}
Figure~\ref{f:deformpolygon} shows a subdiagram of the Hasse diagram for $\dpos$ for $6$-gons.  For the sake of presentation, we have forgone the labeling on the vertices.  A polygonal representative for each equivalence class is drawn, along with its underlying visibility graph.  Each element of $\dpos$ corresponds to a complex $\K_P$ as displayed in Figure~\ref{f:deformpolyhedra}.  Notice that as the polygon deforms and loses visibility edges, its associated polytopal complex collapses into a vertex of $\K_6$.
\end{exmp}

\begin{figure}[h]
\includegraphics[width=.9\textwidth]{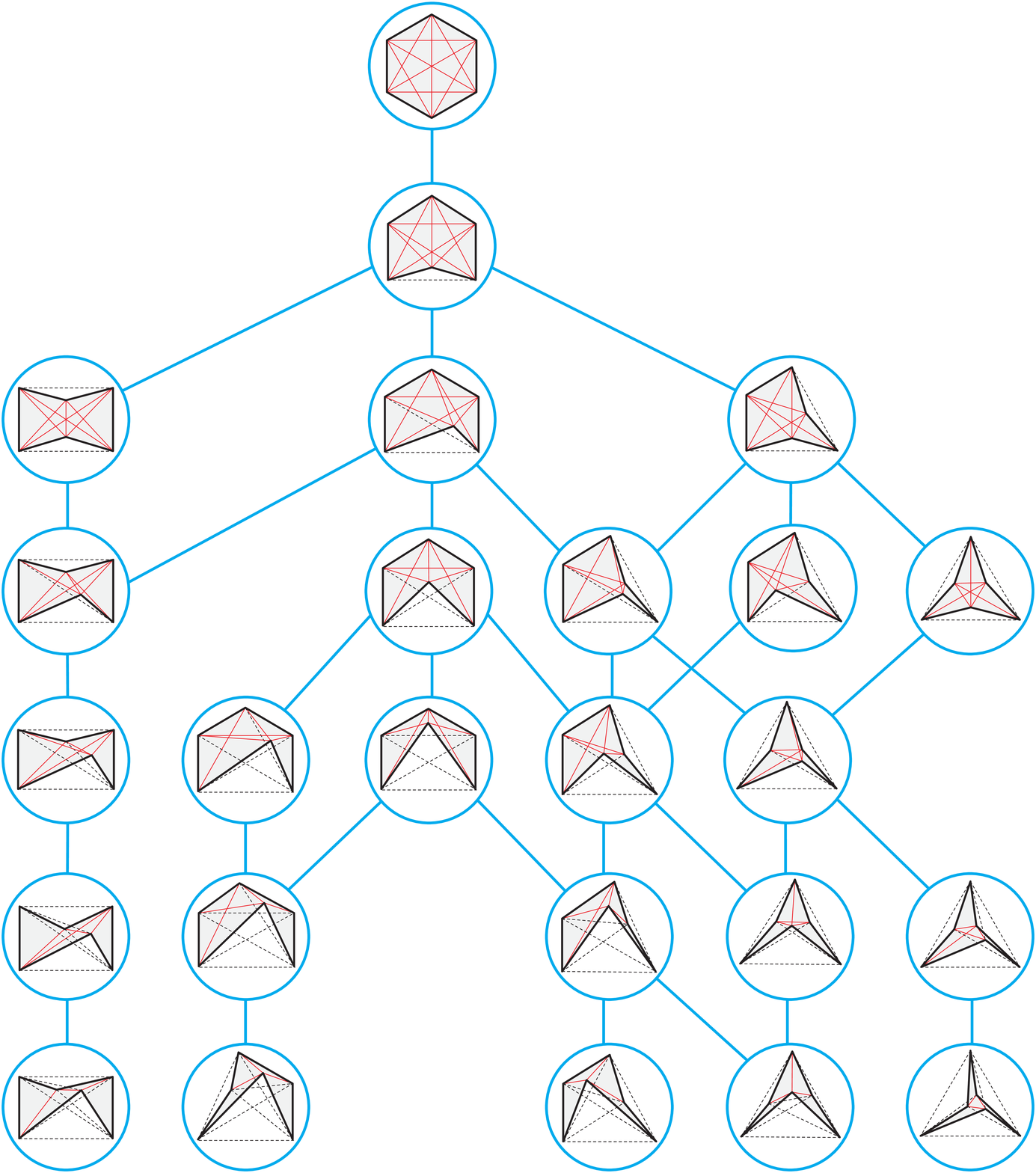}
\caption{A subdiagram of the Hasse diagram of $\dpos$ for $6$-gons.}
\label{f:deformpolygon}
\end{figure}

\begin{figure}[h]
\includegraphics[width=.9\textwidth]{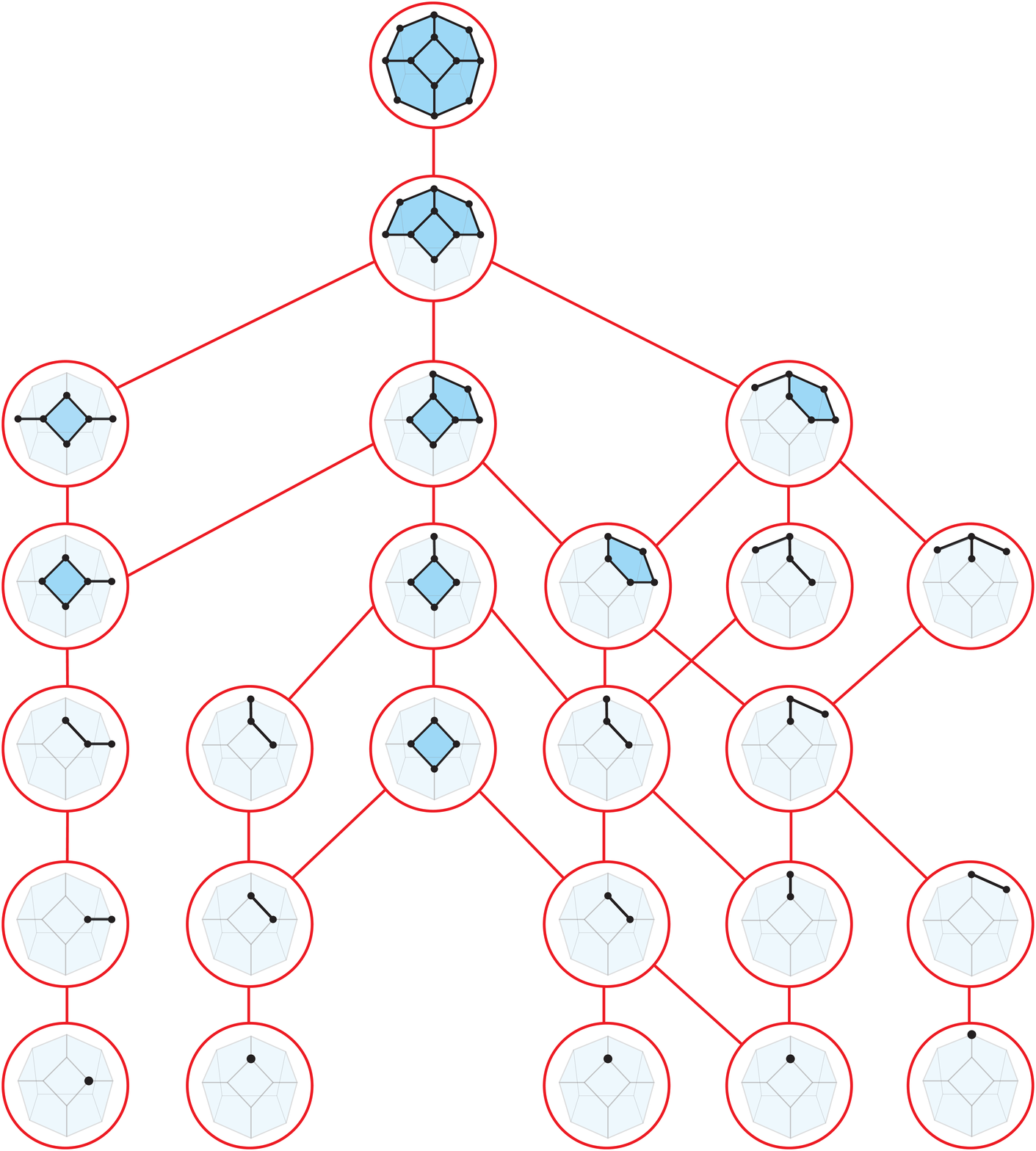}
\caption{The subcomplexes of $\K_6$ associated to the elements of Figure~\ref{f:deformpolygon}.}
\label{f:deformpolyhedra}
\end{figure}

\subsection{}

It is easy to see that the deformation poset $\dpos$ is connected:  Notice that $\dpos$ has a unique \emph{maximum} element corresponding to the convex polygon.  Given any polygon $P$ in the plane, one can  move its vertices, deforming $P$ into convex position, making each element of $\dpos$ connected to the maximum element.  Since the vertices of $P$ are in general position, we can insure that the visibility graph of $P$ changes only by one diagonal at a time during this deformation. However, during this process, the visibility graph of the deforming polygon might gain and lose edges, moving up and down the poset $\dpos$.  

We are interested in the combinatorial structure of the deformation poset beyond connectivity.  The maximum element of $\dpos$ corresponds to the convex $n$-gon (with $\binom{n}{2}$ edges in its visibility graph) whereas the minimal elements (which are not unique in $\dpos$) correspond to polygons with unique triangulations (with $2n-3$ edges in each of their visibility graphs). This implies that the height of the deformation poset is 
$$\binom{n}{2} - 2n + 4.$$
We pose the following problem and close this paper with a discussion of partial results:

\begin{dop}
Show that every maximal chain of $\dpos$ has length $\binom{n}{2} - 2n + 4$.
\end{dop}

We can rephrase this open problem loosely in the deformation context:  Does there exists a deformation of any simple polygon into a convex polygon such that throughout the deformation the visibility of the polygon monotonically increases?  And moreover, does there exists a deformation of any simple polygon into a polygon with a unique triangulation such that throughout the deformation the visibility of the polygon monotonically decreases?  

For the remaining part of the paper, we focus on the monotonically increasing segment of the deformation problem. A natural approach is to discretize this problem into moving vertices of the polygon one by one.  So a stronger claim is as follows:  For any (nonconvex) polygon, there exists a vertex which can be moved in the plane that preserves the visibility of vertices and introduces a new visibility.  In other words, for any polygon, there exists a vertex which can be moved such that we can always move up in rank in the poset.   Figure~\ref{f:counter} provides an elegant counterexample to this claim \cite{aic}, obtained by the participants of the first Mexican workshop on Computational Geometry (\verb+http://xochitl.matem.unam.mx/~rfabila/DF08/+).  A partial collection of the visibility edges of this polygon is given in red.  For this polygon, no vertex can be moved to increase visibility without first losing its current visibility.   

\begin{figure}[h]
\includegraphics[width=.7\textwidth]{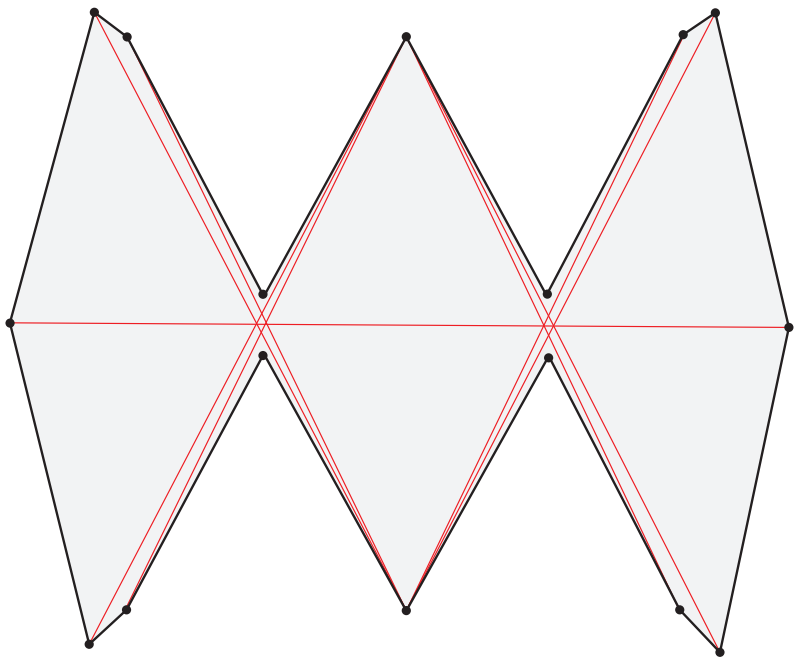}
\caption{No vertex may be moved to increase visibility without first decreasing visibility.}
\label{f:counter}
\end{figure}

We do have the following positive result in the special case of \emph{star} polygons.  A polygon $P$ is a star polygon if there exists a point $p \in P$ such that $p$ is visible to all points of $P$.

\begin {thm}
Let $P$ be a star $n$-gon.  There exists a chain in $\dpos$ from $P$ to the maximum element.
\end{thm}

\begin {proof} 
Let $x$ be a point in the \emph{kernel} of $P$, the set of points which are visible to all points of $P$.   Choose an $\eps$-neighbourhood around $x$ contained in the kernel. For any $a \in P$, let $p(a)$ be a point on the boundary of $P$ which is the intersection of the ray from $x$ passing through $a$ with the boundary.  Let 
$$r(a) \ = \ \frac{d(a,x)}{d(p(a),x)}$$
and let $a'$ be the point on the ray from $x$ passing through $a$ such that $d(a',x)
= \eps \cdot r(a)$. Let $\phi$ be the map from $a$ to $a'$.
We thus construct a linear map $f: P \times [0,1] \rightarrow P$ where $f(P,0) = P$ and $f(P,1) = \phi(P)$ and where 
$$\frac{\partial}{\partial t} f(a, t) = r(a).$$
For any two visible vertices $a$ and $b$ of $P$, consider the triangle $abx$. There cannot be any vertices of $P$ contained in the triangle. If for any vertex $c$ of $P$, the ray from $x$ passing through $c$ intersects the line segment $ab$ at a point $z$, then $d(c,x) > d(z,x)$ and thus for no $t \in [0,1)$ can $d(f(c,t),x) \leq d(f(z,t),x)$. So no visibility is lost during the transformation, but notice that $\phi(P)$ is a circle. However, if we apply $f(a,t)$ only to the vertices of $P$ and map any point $z$ on an edge $(a,b)$ of $P$ to $z'$ on the edge between $f(a,t)$ and $f(b,t)$, we find that we get a polygon at every $t$.   Moreover, the edge is always further from $c$ than $f(z,t)$ for every $t \in [0,1]$, and thus visibility is still maintained.
\end{proof}

%
%
\bibliographystyle{amsplain}

\end{document}